\documentclass[reqno,12pt]{amsart}

\usepackage{amsmath}
\usepackage{amsfonts}
\usepackage{amssymb}
\usepackage{eufrak}
\usepackage{amscd}
\usepackage{amsthm}
\usepackage{epsfig}
\usepackage{amstext}
\usepackage[all]{xy}

\theoremstyle{plain}

\newtheorem{theorem}{Theorem}

\newtheorem{lemma} [theorem] {Lemma}
\newtheorem{remark}[theorem]{Remark}

\usepackage{graphicx}
\usepackage{amsmath}
\usepackage{fullpage}

\newtheorem{corollary}[theorem]{Corollary}

\newtheorem{proposition}[theorem]{Proposition}

\numberwithin{theorem}{section} \numberwithin{equation}{section}

\newcommand{\A}{\mathcal{A}}
\newcommand{\F}{\mathcal{F}}
\newcommand{\R}{\mathcal{R}}

\newcommand{\E}{\mathcal{E}}
\newcommand{\K}{\mathcal{K}}
\newcommand{\N}{\mathcal{N}}

\newcommand{\Y}{\mathcal{Y}}
\newcommand{\X}{\mathcal{X}}
\newcommand{\cP}{\mathcal{P}}

\newcommand{\Ker}{\rm{Ker\;}}
\newcommand{\supp}{\rm{supp\;}}
\newcommand{\cS}{\mathcal{S}}
\newcommand{\mm}{\medskip}

\newcommand{\D}{\mathcal{D}}
\newcommand{\T}{\mathcal{T}}
\newcommand{\cL}{\mathcal{L}}

\newcommand{\I}{\mathcal{I}}

\newcommand{\fA}{\mathfrak{A}}
\newcommand{\m}{\medskip}

\newcommand{\fX}{\mathfrak{X}}
\newcommand{\fN}{\mathfrak{N}}

\newcommand{\M}{\mathcal{M}}

\newcommand{\s}{\subset}
\newcommand{\B}{\mathcal{B}}
\newcommand{\cB}{\mathcal{B}}

\newcommand{\ot}{\otimes}

\begin{document}

\title{On algebras generated by inner derivations}
\author{Tatiana Shulman, Victor Shulman}
\date{\today}
\address{Department of Mathematics, Moscow State Aviation Technological Institute, Orshanskaya 3,
Moscow 121552, Russia}

\email{tatiana$_-$shulman@yahoo.com}
\address{Department of Mathematics, Vologda State Technical University, 15 Lenina Street, Vologda 160000, Russia}

\email{shulman$_-$v@yahoo.com}



\date{}

\maketitle

\begin{abstract} We look for an effective description of the algebra $D_{Lie}(\X,B)$
of operators on a  bimodule $\X$ over an algebra $B$, generated by
all operators $x\to ax-xa$, $a\in B$. It is shown that in some
important examples $D_{Lie}(\X,B)$ consists of all elementary
operators $x\to \sum_i a_ixb_i$ satisfying the conditions $\sum_i
a_ib_i = \sum_i b_ia_i = 0$. The Banach algebraic versions of these
results are also obtained and applied to the description of closed
Lie ideals in some Banach algebras, and to the proof of a density
theorem for Lie algebras of operators on Hilbert space.
\end{abstract}

\section{Introduction}

 Let $B$ be an algebra.  A subspace $C$ of $ B$ is
 called a Lie ideal of $B$ if $ax-xa \in C$ for all $a\in B$, $x\in C$.
The structure of Lie ideals of an associative algebra attracted the
attention of algebraists and Banach-algebraists since seminal works
of Herstein and Jacobson-Rickart (\cite{Hers}, \cite{JR}). A more
general subject is the structure of {\it Lie submodules} in
arbitrary $B$-bimodule $\X$ --- the subspaces of $\X$, defined in
formally identical way. For example Lie ideals of each algebra that
contains $B$ as a subalgebra are Lie submodules over $B$. (This
example is in fact most general: if $\X$ is a $B$-bimodule and $\Y$
is a Lie submodule of $\X$ then one can introduce a product
$(a_1\oplus x_1)(a_2\oplus x_2) = a_1a_2\oplus (a_1x_2+x_1a_2)$ in
$A = B\oplus \X$ and identify $\Y$ with a Lie ideal $0\oplus \Y$ of
$A$).

If one denotes by $L_a$ and $R_a$ respectively the operators of the
left and right multiplication by $a$, then one can say that Lie
submodules are invariant subspaces for the set $\Xi(\X,B)$  of all
{\it inner derivations} $\delta_a = L_a - R_a$, $a\in B$.

Note that the structure of the algebra $\fA$, generated by some set
$\cP$ of operators on a linear space $\X$, gives useful information
about the invariant subspaces of $\cP$. It suffices to say that all
invariant subspaces are sums of cyclic subspaces $\fA x$, $x\in \X$.
Of course, the usefulness of such information depends on the clarity
of the description of $\fA$.

Thus trying to describe the structure of Lie ideals of an algebra
$B$ and Lie submodules of bimodules over $B$, it is natural to
consider the algebras $\D_{Lie}(\X,B)$, generated by the sets
$\Xi(\X,B)$ for various algebras $B$ and bimodules $\X$. We will see
that in some important cases these algebras can be described
effectively: they coincide with the algebras $\M_{Lie}(\X,B)$ of all
elementary operators $ \sum_{k=1}^n L_{a_k}R_{b_k}$ on $\X$,
satisfying the conditions
\begin{equation}\label{dir}
\sum a_kb_k = 0
\end{equation}
and
\begin{equation}\label{conver}
\sum b_ka_k = 0.
\end{equation}
The algebra of all elementary operators on a bimodule $\X$ is
denoted by $El(\X,B)$. For the most popular case $\X = B$, we write
$El(B)$, $\D_{Lie}(B)$ and $\M_{Lie}(B)$ instead of $El(B,B)$,
$\D_{Lie}(B,B)$ and $\M_{Lie}(B,B))$.

 It is convenient and
interesting to consider the corresponding problems in tensor
algebras. Let $B$ be a unital algebra and $B^{op}$ denote the
opposite algebra to $B$ (that is the same linear space with the
reverse multiplication: $a{\ast}b = ba$). Then each bimodule $\X$
over $B$ can be considered as a module over the tensor product
$B{\otimes}B^{op}$, and the map $a\otimes b \to L_aR_b$ extends to a
surjective homomorphism of $B{\otimes}B^{op}$ onto the algebra of
all elementary operators in $\X$. It sends the elements of the form
$a\otimes 1 - 1\otimes a$  to the inner derivations $\delta_a$. Let
$\T_{Lie}(B)$ be the algebra, generated by all elements of the form
$a\otimes 1 - 1\otimes a$, and $\N_{Lie}(B)$ the algebra of all
tensors $\sum_k a_k\ot b_k$ satisfying (\ref{dir}) and
(\ref{conver}). If one proves that
\begin{equation}\label{mainalg}
\T_{Lie}(B) = \N_{Lie}(B) \end{equation}
then one obtains the equality
\begin{equation}\label{eqbim}
\D_{Lie}(\X,B) = \M_{Lie}(\X,B)
\end{equation}
for all $B$-bimodules $\X$.

We consider also the Banach-algebraic versions of the problems.  If
$B$ is a Banach algebra and $\X$ is a Banach $B$-bimodule then all
elementary operators are bounded and one can consider the norm
closures $\overline{\D_{Lie}(\X,B)}$ and $\overline{\M_{Lie}(\X,B)}$
of the algebras $\D_{Lie}(\X,B)$ and $\M_{Lie}(\X,B)$ in the algebra
$\cB(\X)$ of all bounded operators on $\X$. These algebras can
coincide even if $\D_{Lie}(\X,B) \neq \M_{Lie}(\X,B)$.

Furthermore each Banach $B$-bimodule can be considered as a Banach
module over the projective tensor product $V_B =
B\hat{\otimes}B^{op}$. In $V_B$ one can consider the closures of the
algebras $\T_{Lie}(B)$ and $\N_{Lie}(B)$. It is natural to consider
also the algebra $\fN_{Lie}(B)$ of all elements
$\sum_{k=1}^{\infty}a_k\ot b_k\in V_B$ satisfying (\ref{dir}) and
(\ref{conver}) (with the norm-convergence of series). It is not
difficult to see that
\begin{equation}\label{ord}
\overline{\T_{Lie}(B)} \subset \overline{\N_{Lie}(B)} \subset
\fN_{Lie}(B).
\end{equation}

We don't know examples for which $\overline{\N_{Lie}(B)} \neq
\fN_{Lie}(B)$. If $B$ is commutative then the coincidence of these
algebras follows from the identity $\sum_{k=1}^{\infty}a_k\ot b_k =
\sum_{k=1}^{\infty}(a_k\ot b_k - a_kb_k\ot 1)$.

Note that if $B$ is an algebra of functions on a compact $K$
($B\subset C(K)$) then $V_{B} \subset C(K\times K)$ and $\fN_{Lie}$
consists of all functions $f(x,y)\in V_{B}$ for which
$$f(x,x) = 0.$$

Of course if
\begin{equation}\label{baneq}
\overline{\T_{Lie}(B)} = \overline{\N_{Lie}(B)}
\end{equation}
 then
\begin{equation}\label{banel}
\overline{\D_{Lie}(\X,B)} = \overline{\M_{Lie}(\X,B)}
\end{equation}
for each bimodule $\X$.

In Section 2 we consider the case that $B$ is an algebra. It is shown that the equality (\ref{mainalg}) holds
for algebras with one generator, for semisimple finite-dimensional algebras and for the algebras of finite rank
operators on linear spaces. On the other hand, it does not hold for (the algebra of) polynomials of $n>1$
variables, trigonometrical polynomials, rational functions  and for free algebras with $n\ge 2$ generators. The
corresponding results are obtained for elementary operators.

In Section 3 we discuss the equality (\ref{baneq}) for  commutative Banach algebras. Some results in this case
can be obtained from the purely algebraic results of Section 2 by using the fact that if a commutative Banach
algebra $B$ has a dense subalgebra satisfying (\ref{mainalg}) then $B$ satisfies (\ref{baneq}). But in general
the condition (\ref{baneq}) is more "common"~ than (\ref{mainalg}): the completion (in a natural norm) of an
algebra non-satisfying (\ref{mainalg}) can satisfy (\ref{baneq}). The main positive result of the section is
that (\ref{baneq}) is true when $B = C(K)$, the algebra of all continuous functions on a compact $K$. More
generally, (\ref{baneq}) holds for all commutative regular Banach algebras $B$ such that the diagonal is the set
of spectral synthesis for $B\widehat{\otimes} B$.

Section 4 is devoted to the case that $B = \K(\fX)$, the algebra of all compact operators on a Banach space
$\fX$. It should be noted that in the case of non-commutative Banach algebras the validity of (\ref{baneq}) for
a dense subalgebra does not imply its validity for the whole algebra. Hence one cannot deduce (\ref{baneq}) for
$B = \K(\fX)$ from the validity of (\ref{mainalg}) for $\F(\fX)$.

 We establish (\ref{banel}) for $\fX = H$, the
separable Hilbert space, when elementary operators act on $B$ itself. For general $\fX$, a somewhat more weak
than (\ref{baneq}) equality is proved:
$$\overline{\N_{Lie}(B)^2}
= \overline{\T_{Lie}(B)^2}.$$ It implies a weaker version of
(\ref{banel}):
$$\overline{\D_{Lie}(\X,B)^2}
= \overline{\M_{Lie}(\X,B)^2},$$ where the $B$-bimodule $\X$ is
arbitrary.

In Section 5 we obtain some applications to the structure of closed
Lie ideals in algebras $A = B\widehat{\otimes}F$, where $F$ is a
uniformly hyperfinite C*-algebra, $B$ an arbitrary unital Banach
algebra. The main result is that each closed Lie ideal of $A$ is of
the form $$L = I\hat{\otimes}F_{\tau} + M\hat{\otimes} 1$$ where $I$
is a closed ideal of $B$, $M$ is a closed Lie ideal of $B$.

In Section 6 the obtained results are applied to the study of
invariant subspaces of some Lie algebras of operators on a separable
Hilbert space $H$. In particular we establish a Burnside type
theorem for Lie algebras of operators on $H$ that contain maximal
abelian selfadjoint subalgebras of the algebra $B(H)$ of all bounded
operators on $H$.

The authors are grateful to Matej Bresar and Lyudmila Turowska for helpful discussions and very valuable
information.

\section{Algebraic results}

The following  result shows  that $\N_{Lie}(B)$ and $\M_{Lie}(\X,B)$
are "semiideals" of $B\otimes B^{op}$ and $El(\X,B)$, generated
respectively by $\T_{Lie}(B)$ and $\D_{Lie}(\X,B)$.

\begin{lemma}\label{semiideals}

(i) The intersection of one-sided ideals of $B\otimes B^{op}$,
generated by $\T_{Lie}(B)$, coincides with $\N_{Lie}(B)$.

(ii) The intersection of one-sided ideals of $El(\X,B)$, generated
by $\D_{Lie}(\X,B)$, coincides with $\M_{Lie}(\X,B)$.
\end{lemma}

\begin{proof}
Let $J$ be a left ideal of $B\otimes B^{op}$, containing
$\T_{Lie}(B)$. Then each element of the form $a\ot b - 1\ot ab$
belongs to $J$, because it equals to $(1\ot b)(a\ot 1-1\ot a)$.
Hence $J$ contains the set $J_1$ of all elements $\sum a_i\ot b_i$
with $\sum a_ib_i = 0$, because each of them can be written in the
form $\sum (a_i\ot b_i - 1\ot a_ib_i)$.  It follows that $J_1$ is
the left ideal of $B\otimes B^{op}$, generated by $\T_{Lie}(B)$.
Similarly the right ideal generated by $\T_{Lie}(B)$ coincides with
$J_2 = \{\sum a_i\ot b_i: \sum b_ia_i = 0\}$. Since $\N_{Lie}(B) =
J_1\cap J_2$, this proves (i). The proof of (ii) is similar.
\end{proof}

\medskip

As a consequence we get the inclusions
\begin{equation}\label{mulby1}
\T_{Lie}(B) (B\ot B^{op}) \T_{Lie}(B) \s \N_{Lie}(B)
\end{equation}
and
\begin{equation}\label{mulby2} \D_{Lie}(\X,B)El(\X,B)\D_{Lie}(\X,B) \s
\M_{Lie}(\X,B).
\end{equation}

 We consider the problem of validity of the equality
(\ref{mainalg}) for arbitrary associative algebra $B$. Let $\frak L$
denote the class of all algebras, for which this equality is true.

\begin{lemma}\label{genprop}
 If $B$ is a commutative unital algebra in $\frak L$, then each
quotient of $B$ belongs to $\frak L$.
\end{lemma}
\begin{proof}
Let $I$ be an ideal in $B$, $C = B/I$ and $\pi:B\to C$ the canonical
epimorphism.

Let $f = \sum_k u_k\ot v_k \in \N_{Lie}(C)$. For each $k$ choose
$a_k,b_k\in B$ such that $\pi(a_k) = u_k$, $\pi(b_k) = v_k$ and set
$g = \sum a_k\ot b_k$. Since $\sum_k u_kv_k = 0$ we have that
$c:=\sum_k a_kb_k \in I$.

The element $g-c\ot 1$ belongs to $\N_{Lie}(B)$. Hence it belongs to
$\T_{Lie}(B)$. Since $\pi\ot \pi$ clearly sends $\T_{Lie}(B)$ to
$\T_{Lie}(C)$, we get that $f = \pi\ot\pi(g-c\ot 1) \in
\T_{Lie}(C)$. This is what we need.
\end{proof}

\medskip

\begin{theorem}\label{polyn}
Each algebra $B$ with one generator belongs to $\frak L$.
\end{theorem}
\begin{proof}
Since each algebra with one generator is a quotient of the algebra
of polynomials in one variable it follows from Lemma \ref{genprop}
that we may restrict to the case when $B$ is the algebra of
polynomials in one variable. Clearly $B{\otimes}B$ can be identified
with the algebra $\cP_2$ of polynomials in two variables. So in this
presentation ${\T_{Lie}(B)}$ is the subalgebra in $\cP_2$, generated
by all polynomials of the form $p(x)-p(y)$. It can be easily checked
that ${\N_{Lie}(B)}$ coincides with the algebra $J$ of all
polynomials $p(x,y)$ satisfying the equality $p(x,x) = 0$. It is not
difficult to see that $J$ is the ideal of $\cP_2$ generated by the
polynomial  $x-y$. Let $J_n$ be the set of all uniform polynomials
of degree $n$ in $J$; since $x-y$ is uniform, $J = \sum_n J_n$. It
is clear that a polynomial $p(x,y) = \sum_{i=0}^na_ix^iy^{n-i}$
belongs to $J_n$ if and only if  $\sum_ia_i = 0$.

We will show by induction that $J_n\subset \T_{Lie}(B)$. For $n = 1$
this is evident.

Suppose this is proved for $n< k$. Let $e_1,...,e_k$ be a basis in
$J_k$, and let $e(x,y) = x^k-y^k$. Then for each $i$ there is
$\lambda = \lambda_i$ such that $(x-y)^2$ divides $e_i - \lambda_i
e$. Indeed setting $u_i(x,y) = e_i(x,y)/(x-y)$, $u(x,y) =
e(x,y)/(x-y)$ we see that $u\notin J$, or in other words $s(u) \neq
0$, where $s(u)$ is the sum of coefficients of $u$. So it suffices
to take $\lambda_i = s(u_i)/s(u)$.

It follows that the functions $u_i - \lambda_i u$ belong to
$J_{k-1}$. By induction hypothesis, they belong to $\T_{Lie}(B)$.
Hence $e_i-\lambda_i e\in \T_{Lie}(B)$. Since also $e\in
\T_{Lie}(B)$, we get that all $e_i$ are in $\T_{Lie}(B)$. Thus
$J_k\subset \T_{Lie}(B)$.
\end{proof}

\medskip

Dealing with Lie ideals of associative algebras it is natural to
put the question: which algebraic expressions in elements
$x_1,...,x_n$ of an algebra one can write, being guaranteed that
their results are in the given Lie ideal, containing all $x_i$? A
more correct formulation: which elementary operators preserve Lie
ideals? Using Theorem \ref{polyn}, we obtain the answer in a
simplest case.

\begin{corollary}\label{simplest}
Let a matrix $(\lambda_{k,m})_{k,m\in \mathbb{N}}$ be given with
only finite number of non-zero entries. Suppose that
 \begin{equation}\label{conv}
 \sum_{k+m =
 n}\lambda_{k,m} = 0\;{\rm  \;\; for\; all\;}n.
 \end{equation}
If $\cL$ is a Lie ideal of
 an associative algebra $\B$ then  $\sum_{k,m}\lambda_{k,m}a^kba^m \in \cL$,
 for all $a\in \B$, $b\in \cL$.
 \end{corollary}
 \begin{proof}
Clearly $\sum_{k,m}\lambda_{k,m}a^kba^m = P(L_a,R_a)b$ where
$P(\alpha,\beta) = \sum_{k,m}\lambda_{k,m}\alpha^k\beta^m$. The
condition (\ref{conv}) provides that $P(\alpha,\alpha) = 0$. It
follows from Theorem \ref{polyn} that $P$ belongs to the subalgebra
 generated by all polynomials $p(\alpha) - p(\beta)$. Since $L$ is
invariant under $p(L_a) - p(R_a) = L_{p(a)} - R_{p(a)}$, it is
invariant under $P(L_a,R_a)$.
\end{proof}

\medskip

 It can also be proved that the condition (\ref{conv}) is necessary in
 some sense. Namely, if $\lambda_{k,m}$ don't satisfy it then there
 are an algebra $\A$, its  Lie ideal $\cL$ and elements $a,b\in \cL$ such
 that
 $\sum_{k,m}\lambda_{k,m}a^kba^m \notin \cL$.

Indeed, setting $\mu_n = \sum_{k+m = n}\lambda_{k,m}$ we have from
 the above that $\sum_{k,m}\lambda_{k,m}a^kba^m - \sum_n \mu_na^nb
 \in \cL$. So it suffices to construct $\cL$ in such a way that $\sum_n \mu_na^nb
 \notin \cL$. It is easy to show that if $p$ is a polynomial of
 degree $\ge 1$ and $\A$ is commutative algebra then there is a Lie ideal
  $\cL$ of  $\A$ (any subspace is a Lie ideal) and
 elements $a,b\in \cL$ with $p(a)b\notin \cL$ (take $b = a$, $\cL = {\mathbb C}a
 $).

\medskip

To demonstrate possible applications of the results of such kind let
us consider one of the simplest examples: the algebra $M_n$ of all
$n\times n$ matrices as a bimodule over the algebra $D_n$ of all
diagonal matrices, with respect to the matrix multiplications.

For a subset $K$ of $\{1,2,...,n\}\times\{1,2,...,n\}$, denote by
$Z(K)$ the space of all matrices $a\in M_n$ with $a_{jk} = 0$ for
$(j,k)\notin K$.

\begin{corollary}\label{subbimmat}
Each  Lie $D_n$-submodule of $M_n$ is a direct sum $S = G + Z(K)$
where $G$ is a subspace of $D_n$  and $K$ is a subset of
$\{1,2,...,n\}\times\{1,2,...,n\}$, non-intersecting the diagonal.
\end{corollary}
\begin{proof}
Lie submodules are invariant subspaces for the algebra
$\T_{Lie}(D_n)$. Clearly $D_n$ is generated by one element, hence,
by Theorem \ref{polyn},  $\T_{Lie}(D_n)$ coincides with
$\N_{Lie}(D_n)$ which can be realized as the algebra of all matrices
with zero's on the diagonal (and pointwise multiplication). In this
realization the action of $\N_{Lie}(D_n)$ on $M_n$ is also the
pointwise (Hadamard) multiplication. Then in fact we have the action
of the algebra $\mathbb{C}^{n(n-1)}$ on the direct sum $D\oplus
\mathbb{C}^{n(n-1)}$, and the action on the first summand is
trivial. It follows that the invariant subspaces are direct sums of
subspaces of $D$ and "coordinate" subspaces of
$\mathbb{C}^{n(n-1)}$. This proves our assertion.
\end{proof}

\medskip

As a consequence we get a well known  (see for example \cite{Hers}
where a similar result was obtained for matrices over arbitrary
ring) description of Lie ideals in $M_n$:

\begin{corollary}\label{standard}
The only Lie ideals in $M_n$ are $0$, $M_n$, $\mathbb{C}1$ and the
space $M_n^0$ of all matrices with zero trace.
\end{corollary}
\begin{proof}
Any Lie ideal $S$ of $M_n$ is a Lie submodule over $D_n$. Hence it
has the form  $S = G + Z(K)$, where $G\s D_n$. Hence
$$[a,Z(K)] \s Z(K) + D_n , \text{  for all }a\in M_n.$$
The multiplication table of matrix units shows that this condition
holds if either $K = \emptyset$ or $K = \{(j,k): j\neq k\}$. In the
first case $S = G\s D_n$. If $a_{ii}\neq a_{jj}$ for some $a\in S$,
then $[e_{ij},a]\notin D_n$, where $e_{ij}$ is the corresponding
matrix unit. Hence $S$ consists of matrices $\lambda 1$, so either
$S = 0$ or $S = \mathbb{C}1$.

If  $K = \{(j,k): j\neq k\}$ then $e_{ii}-e_{jj} = [e_{ij},e_{ji}]
\in S$, for all $(i,j)$ with $i\neq j$. Hence $G$ contains the
linear span of all $e_{ii}-e_{jj}$ which coincides with the space of
all diagonal matrices of zero trace. It follows that either $G =
D_n$ or $G = D_n \cap M_n^0$. So $S = M_n$ or $S = M_n^0$.
\end{proof}

\m

Now let us show that (\ref{mainalg}) does not hold for all (even for
all commutative) algebras.

\begin{theorem}\label{counterex}
The equality (\ref{mainalg}) is not true when

(i) $B = \cP_2$, the algebra of polynomials in two variables $(x_1,x_2)$,

(ii) $B = \cL$, the algebra of Loran polynomials,

(iii) $B = \R_1$, the algebra of rational functions of one variable.
\end{theorem}
\begin{proof}
(i) Clearly we can identify $B{\otimes}B$ with the algebra $\cP_4$ of polynomials in four variables $(\vec
x,\vec y) = (x_1,x_2,y_1,y_2)$ by the equality $(p\otimes q)(\vec x,\vec y)= p(\vec x)q(\vec y)$.    The
polynomial $p(\vec x, \vec y) = (x_1-y_1)x_2$ belongs to ${\N_{Lie}(B)}$ because $p(x_1,x_2,x_1,x_2) = 0$. We'll
prove that it does not belong to ${\T_{Lie}(B)}$.

Assume the contrary, then
\begin{equation}\label{temp1}
p(\vec x,\vec y) = \sum_i q_i(\vec x,\vec y) \end{equation}
 where each $q_i$ is the product of polynomials of the form $a(\vec x)-a(\vec y)$. It
is evident that one can assume that all $a(\vec x)$ are monomials:
$a(\vec x) = x_1^kx_2^m$. Hence each $q_i$ is a uniform polynomial.
Taking only the uniform polynomials of degree 2 in the right hand
side of (\ref{temp1}), we see that
\begin{multline}\label{long}
p(\vec x,\vec y) = \lambda_1(x_1-y_1)^2 +
\lambda_2(x_1-y_1)(x_2-y_2) +\\ \lambda_3(x_2-y_2)^2 +
\lambda_4(x_1^2-y_1^2)+ \lambda_5(x_1x_2-y_1y_2)
 + \lambda_6(x_2^2-y_2^2) .
\end{multline}
Setting $x_1 = y_1$ we get $$ 0 = \lambda_3(x_2-y_2)^2 +
\lambda_5x_1(x_2-y_2)
 + \lambda_6(x_2^2-y_2^2).$$
 Hence $$ 0 = \lambda_3(x_2-y_2) + \lambda_5x_1
 + \lambda_6(x_2+y_2) ,$$
 which easily implies  that
 $\lambda_3 = \lambda_5 = \lambda_6 = 0.$
 So (\ref{long}) gives:
 $$x_2(x_1-y_1) = \lambda_1(x_1-y_1)^2 + \lambda_2(x_1-y_1)(x_2-y_2)
 + \lambda_4(x_1^2-y_1^2).$$ It follows that $$x_2 =  \lambda_1(x_1-y_1) + \lambda_2(x_2-y_2)
 + \lambda_4(x_1+y_1).$$ Setting $x_1 = y_1$, $x_2 = y_2$ we get $x_2
 = 0$, a contradiction.

 To prove (ii) and (iii) we need the following auxiliary statement.

 \begin{lemma}\label{vspom}
 Let  $\R_2$ be  the algebra of all rational functions in two variables, and $\E$ --- the subalgebra of $\R_2$,
 generated by all functions $f(x) - f(y)$, $f\in \R_1$. Then the function $g(x,y) = \frac{x-y}{x}$ does not belong to $\E$.
 \end{lemma}
 \begin{proof}
 Assuming the contrary we have the equality
 \begin{equation}\label{twosum}
 g(x,y) = f(x) - f(y) + h(x,y),
 \end{equation}
 here $h(x,y)$ belongs to the linear span $U$ of all elements of $\R_2$ of the form
 \begin{equation}\label{U}
u(x,y) = \prod_{k=1}^n
 (f_k(x)-f_k(y))\text{ with }n\ge 2, f_i\in \R_1.
 \end{equation}
 We divide both parts of (\ref{twosum}) by $x-y$ and, fixing $x$ in the domains of definition of all functions $f_k$
 in all products of the form (\ref{U}) that participate in $h(x,y)$,
  take the limit for  $y\to x$.

  Note that for $u(x,y)$ of the form (\ref{U}), one has $\lim_{y\to x}\frac{u(x,y)}{x-y} = 0$. Hence we obtain the equality
  $$\frac{1}{x} = f^{\prime}(x)$$
  which is impossible because $f$  is rational.
  \end{proof}
  To finish the proof of parts (ii) and (iii) of Theorem \ref{counterex}, note that for $B = \R_1$ and $B = \cL$,
  the function
  $g(x,y)= \frac{x-y}{x}$ belongs to the ideal of $B$ generated by all functions of the form $f(x)-f(y)$. But
  Lemma \ref{vspom} shows that it does not belong to the subalgebra generated by these functions.
  \end{proof}

  Note that the algebra of Loran polynomials is isomorphic to the algebra of trigonometrical polynomials. Thus the
equality (\ref{mainalg}) does not hold for the latter.

 \medskip

 \begin{remark}\label{more}
 In the proof of Theorem \ref{counterex}(i) we established that the polynomial
 $(x_1-y_1)x_2$  does not belong to $\T_{Lie}(\cP_2)$. In what
 follows we will need a more strong result: the polynomial
 $(x_1-y_1)^2x_2$ also does not belong to $\T_{Lie}(\cP_2)$.
 \end{remark}
 \begin{proof}
 Suppose that $p(x_1, x_2, y_1, y_2)=(x_1-y_1)^2x_2$ is in $\T_{Lie}(\cP_2)$ and
hence can be represented in the form $\sum_i q_i(x_1, x_2, y_1,
y_2)$ where every $q_i$ is a product of polynomials of the form
$a(x_1, x_2)-a(y_1, y_2)$. We may assume that every $a$ is a
monomial. Since $p(x_1, x_2, x_1, y_2)=0$, it is not hard to see
that the equality must have the form
$$ p(x_1, x_2, y_1, y_2)=
\lambda_1(x_1^3-y_1^3)+\lambda_2(x_1-y_1)(x_1x_2-y_1y_2)+\lambda_3(x_1^2-y_1^2)(x_2-y_2)$$
$$+\lambda_4(x_1-y_1)(x_2^2-y_2^2)+\lambda_5(x_1-y_1)^2(x_2-y_2)+\lambda_6(x_1-y_1)^3+\lambda_7(x_1-y_1)(x_2-y_2)^2.$$
Dividing the both sides by $x_1-y_1$ we get
\begin{multline}\label{temp2}
(x_1-y_1)x_2 =
\lambda_1(x_1^2+x_1y_1+y_1^2)+\lambda_2(x_1x_2-y_1y_2)+\lambda_3(x_1-y_1)(x_2-y_2)\\
+\lambda_4(x_1-y_1)(x_2+y_2)+\lambda_5(x_1+y_1)(x_2-y_2)+\lambda_6(x_1-y_1)^2+\lambda_7(x_2-y_2)^2.
\end{multline}
For $y_1=x_1$ we obtain
$$3\lambda_1x_1^2+\lambda_2x_1x_2-\lambda_2x_1y_2+2\lambda_5x_1x_2-2\lambda_5x_1y_2+\lambda_7x_2^2-2\lambda_7x_2y_2+
\lambda_7y_2^2=0$$ whence $\lambda_1=\lambda_7=0, \lambda_2=-2\lambda_5$ and \begin{multline*} (x_1-y_1)x_2 =
-2\lambda_5(x_1x_2-y_1y_2)+\lambda_3(x_1-y_1)(x_2-y_2)+\\
\lambda_4(x_1-y_1)(x_2+y_2)+\lambda_5(x_1+y_1)(x_2-y_2) +\lambda_6(x_1-y_1)^ 2.\end{multline*}   For $x_1=x_2=x$
and $y_1=y_2=y$ it gives us
$$(x-y)x = -2\lambda_5(x^2-y^2)+\lambda_3(x-y)^2+\lambda_4(x-y)(x+y)+\lambda_5(x+y)(x-y)
+\lambda_6(x-y)^2$$ whence we obtain 3 equations :
$\lambda_3+\lambda_4-\lambda_5+\lambda_6 = 0$,
$\lambda_3-\lambda_4+\lambda_5+\lambda_6=0$, $\lambda_3+\lambda_6 =
0$. This system has the solution of the form $\lambda_3 = \alpha,
\lambda_4 = \beta, \lambda_5 = \beta, \lambda_6 = - \alpha$.
Substituting into (\ref{temp2}) we get the equality $$(x_1-y_1)x_2 =
\alpha((x_1-y_1)(x_2-y_2) - (x_1-y_1)^2)$$ which is a contradiction.
 \end{proof}

\medskip

\begin{corollary}\label{free}
 The equality (\ref{mainalg}) is not true when $B =
\F_2$, the free algebra in two generators $a,b$. In particular the
element $(a\ot 1 - 1\ot a)b\ot 1 (a\ot 1 - 1\ot a)$ does not belong
to $\T_{Lie}(F_2)$.
\end{corollary}
\begin{proof} Let $\pi: \F_2\to \cP_2$ be the homomorphism
which sends the first generator $a$ of $\F_2$ to $x_1$ and the
second generator $b$ --- to $x_2$. Denote by $\pi\otimes \pi $ the
corresponding homomorphism from $\F_2\otimes \F_2^{op}$ to the
algebra $\cP_2 \otimes \cP_2 = \cP_4$. Since $\pi\otimes \pi
(u\otimes 1 - 1\otimes u) = \pi(u)\ot 1 - 1\ot \pi(u) $ and $\pi$ is
surjective, we have the equality $\pi\otimes \pi(\T_{Lie}(\F_2)) =
\T_{Lie}(\cP_2)$.

Set $z = (a\ot 1 - 1\ot a)b\ot 1 (a\ot 1 - 1\ot a)$. This element
belongs to $\N_{Lie}(\F_2)$ by (\ref{mulby1}). On the other hand if
$z$ belongs to $\T_{Lie}(\F_2)$ then $\pi\otimes \pi (z) \in
\T_{Lie}(\cP_2) $. But this contradicts to Remark \ref{more},
because $\pi\otimes \pi (z) = (x_1-y_1)^2x_2 \notin \T_{Lie}(\cP_2)
$.
\end{proof}

\medskip

On the other hand, the equality (\ref{mainalg}) turns out to be true
for some important non-commutative examples.

\begin{theorem}\label{tanya}
The equality (\ref{mainalg}) holds if $B$ is an arbitrary semisimple finite-dimensional algebra.
\end{theorem}
\begin{proof}
It is easy to check that the class $\frak L$ is closed under forming direct sums. Thus, by the Wedderburn's
Theorem, it suffices to prove the equality (\ref{mainalg}) for $B = M_n(\mathbb{C})$, the algebra of all complex
$n\times n$-matrices.

Denote, for brevity,  $\N_{Lie}(B)$ by $\N$ and $\T_{Lie}(B)$ by $\T$. Let $\pi$ be the representation of
$B\otimes B^{op}$ on the space $B$, defined by the equality: $\pi(a\ot b)(x) = axb$. Then $\pi$ is irreducible
and faithful (because $B$ and $B\otimes B^{op}$ are simple). So it suffices to show that $\pi(\T) = \pi(\N)$.

Set $H_1 = \mathbb{C}1$ and $H_2 = \{x\in M: tr (x) = 0\}$. These
subspaces are invariant for $\pi(\N)$ (hence for $\pi(\T))$.
Moreover
$$H_1 = \Ker \pi(\N) = \Ker \pi(\T)$$
and
$$H_2\supseteq \pi(\N)B \supseteq \pi(\T)B.$$
Indeed it is easy to see that $H_1\s \Ker \pi(\N)\s \Ker \pi(\T) =
H_1$. Moreover if $T = \sum_ia_i\ot b_i \in \N$ then $tr(\pi(T)x)
= tr \sum_ia_ixb_i = tr \sum_ib_ia_i x = 0$, for each $x\in B$, so
the range of $\pi(\N)$ is contained in $H_2$.

Let $\cS$ denote the restriction of the algebra $\pi(\T)$ to $H_2$.
Then each non-trivial invariant subspace of the algebra $\cS$ is a
non-zero  Lie ideal of $B$ strictly contained in $H_2$. By Corollary
\ref{standard}, $B$ has no such Lie ideals. So $\cS$ has no
non-trivial invariant subspaces; by Burnside's Theorem, $\cS$
coincides with the algebra $L(H_2)$ of all operators on $H_2$. Hence
the restriction of the algebra $\pi(\N)$ to $H_2$ is also $L(H_2)$.
Since $H_1 = \Ker \pi(\T) = \Ker \pi(\N)$ and the space $B$ is a
direct sum of $H_1$ and $H_2$, we conclude that $\pi(\N) = \pi(\T)$
and $\N = \T$.
\end{proof}

\medskip
Below we always denote by $\widetilde{B}$ the unitization of $B$.

Let $X$ be a linear space, $\cL(X)$ the algebra of all linear
operators on $X$,  $F(X)$ --- the algebra of all finite-rank
operators on $X$.  The algebra $\widetilde{F(X)}$ in this case can
be realized as $F(X) + \mathbb{C}1 \s \cL(X)$. Our next aim is to
show that for the  algebra $\widetilde{F(X)}$ the equality
(\ref{mainalg}) holds.

\begin{lemma}\label{ident} Let $B$ be a unital algebra. Then

(i) For any $x\in B$, the element $1\ot x^2 -x\ot x$ belongs to
$\T_{Lie}(B)$;

(ii) If $a,x\in B$ and $ax = xa = 0$ then
\begin{equation}\label{form}
a\ot x^3 = (1\ot x^2 - x\ot x)(a\ot 1 - 1\ot a)(x\ot 1 - 1\ot x)
\end{equation}
whence $a\ot x^3 \in \T_{Lie}(B)$.
\end{lemma}
\begin{proof}
Part (i) follows from the equality $2(1\ot x^2 -x\ot x) = (x\ot 1
-1\ot x)^2 - (x^2\ot 1 - 1\ot x^2)$. The formula (\ref{form}) can
be checked by easy calculation; using (i) this implies part (ii).
\end{proof}

\medskip

\begin{theorem} \label{Ftens}
The equality (\ref{mainalg}) holds if $B = \widetilde{F(X)}$.
\end{theorem}
\begin{proof}
Note first of all that for finite-dimensional $X$ the result
immediately follows from Theorem \ref{tanya}. So we have to
consider only the case when $dim(X) = \infty$.

An arbitrary element of $B\otimes B^{op}$ can be written in the
form $R = \lambda + a\ot 1 + 1\ot b + \sum _{i=1}^n a_i\ot b_i$,
with $a, b, a_i, b_i \in F(X)$. If $R\in \N_{Lie}(B)$ then it is
easy to see that $\lambda = 0$.

Let $a\in F(X)$ and let $p$ be a finite-rank projection in $F(X)$
such that $ap = pa = a$. Setting $q = 1-p$ we get $aq = qa = 0$. By
Lemma \ref{ident}, $a\ot q^3 \in \T_{Lie}(B)$. Since $q^3 = q$ we
see that $a\ot p - a\ot 1 \in \T_{Lie}(B)$. Similarly $1\ot b - p\ot
b \in \T_{Lie}(B)$ for an appropriate projection $p\in F(X)$. It
follows that modulo $\T_{Lie}(B)$ each element of $\N_{Lie}(B)$ can
be written in the form
$$R = \sum_{i=1}^n a_i\ot b_i$$
with $a_i, b_i \in F(X)$.

Let now $p$ be a finite rank projection such that $a_ip = pa_i =
a_i$ and $b_ip = pb_i = b_i$ for all $i$. Then all $a_i$ and $b_i$
can be considered as operators on finite-dimensional space $Y =
pX$. By Theorem \ref{tanya}, $R$ belongs to the algebra
$\T_{Lie}(\cL(Y))$. But  the natural imbedding of $\cL(Y)$ into
$\cL(X)$ maps $\T_{Lie}(\cL(Y))$ into $\T_{Lie}(F(X)) =
\T_{Lie}(B)$. We conclude that $R\in \T_{Lie}(B)$.
\end{proof}

\medskip

Turning to elementary operators we have the problem of the validity
 of the equality
\begin{equation}\label{mainelem}
\D_{Lie}(B,\X) = \M_{Lie}(B,\X).
\end{equation}

 It is straightforward that if for an algebra $\widetilde{B}$  the
 equality (\ref{mainalg}) holds then (\ref{mainelem}) is also true.
 As a consequence we obtain

 \begin{corollary}
 For algebras $M_n$, $\widetilde{F(X)}$ and for each algebra with one generator, the equality (\ref{mainelem})
 holds.
 \end{corollary}

The next result extends part (ii) of Theorem \ref{counterex}.

\begin{theorem} $\M_{Lie}(\F_2)\neq \D_{Lie}(\F_2)$.
\end{theorem}
\begin{proof}  Let $f: \F_2\otimes \F_2^{op} \to
El(F_2)$ be the standard representation of $\F_2\otimes \F_2^{op}$
by elementary operators on $\F_2$. By Theorem 2.3.13 of \cite{BMM},
it is injective (because $\F_2$ is centrally closed, by Theorem
2.4.4 of \cite{BMM}); the surjectivity of $f$ is obvious. It follows
that $f(\T_{Lie}(\F_2)) = \D_{Lie}(\F_2)$ and $f(\N_{Lie}(\F_2)) =
\M_{Lie}(\F_2)$; using Corollary \ref{free}, we conclude that
$\D_{Lie}(\F_2)\neq M(\F_2)$.
\end{proof}

 \mm

\section{Commutative Banach algebras}

Let now $B$ be a Banach algebra.  A natural Banach-algebraic
analogue of (\ref{mainalg}) is the equality (\ref{baneq}) which for
commutative $B$ is equivalent to
\begin{equation}\label{banalg}
\overline{\T_{Lie}(B)} = \fN_{Lie}(B).
\end{equation}
We are going to consider the question of the validity of these
equalities for different Banach algebras.

 Let us firstly list
some consequences of Theorem \ref{polyn}.

Recall that if $B$ is a function algebra on a compact $K$, such
that $\|f\|_B \ge \|f\|_{C(K)}$ then the natural embedding of
$B\hat{\otimes} B$ into $C(K\times K)$ is injective, so
$B\hat{\otimes}B$ can be considered as a subalgebra of $C(K\times
K)$.
\begin{corollary}\label{disk}
Let $B$ be an algebra of functions on a compact $K\subset \mathbb
C$, supplied with a complete norm in which polynomials are dense in
$B$.   Then $\overline{\T_{Lie}(B)}$ coincides with the ideal $J =
\{f(x,y)\in B\hat{\otimes} B: f(x,x) = 0\}$.
\end{corollary}
\begin{proof}
Let $f(x,y) = \sum_{i=1}^{\infty} a_i(x)b_i(y) \in J$, then $f(x,y)
= \sum_{i=1}^{\infty}(a_i(x)-a_i(y))b_i(y)$ and the series of norms
converges. Hence it suffices to show that
\begin{equation}\label{belongs}
(a(x)-a(y))b(y) \in \overline{\T_{Lie}(B)}.
\end{equation}
 Let firstly $b$ be a
polynomial. The set of all $a(x)$ for which (\ref{belongs}) holds,
is closed in $\B$. It contains all polynomials by Theorem
\ref{polyn}. Hence it coincides with $\B$. Thus (\ref{belongs})
holds for each $a\in \B$ and each polynomial $b$. Since the set of
all $b$, for which (\ref{belongs}) holds with given $a$, is
closed, the condition (\ref{belongs}) holds for all $a,b\in \B$.
\end{proof}

\medskip

 As example for $\B$, one can take $C(0,1)$, or $C^p(0,1)$, or the
disk algebra or, more generally, the closure of polynomials in $C(K)$, for arbitrary $K$,  or the algebra of
absolutely convergent Taylor series on $\mathbb D$.

\medskip

{\bf Problem}. Is the result of Corollary \ref{disk} true for the
algebra $A(K)$ or $R(K)$, where $K\subset \mathbb C$ is arbitrary
compact? Here $A(K)\subset C(K)$ is the algebra of all functions
on $K$, analytical on $int(K)$, $R(K)$ --- the closure of the
algebra of rational functions with poles outside $K$.

\medskip

 Let us look what Corollary \ref{disk} gives for the case of Lie
 submodules in Banach bimodules.

 Denote by $\frak T(\mathbb{D})$ the algebra of all absolutely converging Taylor
 series on $\mathbb D$, that is all functions
 $f(z) = \sum_{k=1}^{\infty}\gamma_k z^k$ with $\|f\|_{\frak T}: = \sum_k|\gamma_k| <
 \infty$. It is clear that the functions in $\frak T(\mathbb{D})$ can be applied to
 any operator $A$ of norm $\le 1$, and that $\|f(A)\|\le
 \|f\|_{\frak T}$. Hence for any function $f$ in the algebra $\cS = \frak{T}(\mathbb
 D)\hat{\otimes}\frak{T}(\mathbb D)$ and any two commuting operators $A$,
 $B$ with norms $\le 1$, one can calculate $f(A,B)$ and
 $\|f(A,B)\|\le \|f\|_{\cS}$.

 \begin{corollary}
 Let $\X$ be a Banach bimodule over a Banach algebra $\A$. Let
 $\cL$ be a closed Lie submodule in $\X$. If $a\in \A$, $\|a\|\le
 1$, then for each function $f(\lambda,\mu)$ in $\cS = \frak{T}(\mathbb
 D)\hat{\otimes}\frak{T}(\mathbb D)$ with $f(\lambda,\lambda) = 0$,
 the operator $f(L_a,R_a)$ leaves $\cL$
 invariant.
 \end{corollary}

\medskip

Our next aim is to show that for $B = C(K)$ the result of
Corollary \ref{disk} holds without the assumption $K\subset
\mathbb C$.

Let $K$ be an arbitrary compact. The Banach algebra $V(K) =
C(K)\hat{\otimes}C(K)$ is called the Varopoulos algebra of $K$. It
is naturally realized as a regular symmetric function algebra on
$K\times K$. The theory of such algebras and their relations to
various branches of analysis was developed in \cite{Var}.

\begin{theorem}\label{Lie}
The closed subalgebra in $V(K)$, generated by all functions
$f(x)-f(y)$, coincides with the ideal of all functions $F(x,y)$,
vanishing on the diagonal: $$\overline{\T_{Lie}(C(K))} = \{F\in
V(K): F(x,x) = 0 {\rm \;for\;all\;}x\in K\}.$$
\end{theorem}
\begin{proof}
The inclusion $\subset$ is evident; we have to prove $\supset$.

\noindent

Let us fix two non-intersecting open subsets $V_1, V_2$ of $K$. We claim that if $supp(f) \subset V_1$ and
$supp(g)\subset V_2$  then $f(x)g(y)\in {\T_{Lie}(C(K))}$.

To prove the claim, set $J_i = \{f\in C(K): supp(f)\subset V_i\}$, for $i=1,2$. By Lemma \ref{ident},
$g(x)f(y)^3\in {\T_{Lie}(C(K))}$ for any two functions $f,g$ such that $f(x)g(x) = 0$. Hence $f(x)^3g(y)\in
{\T_{Lie}(C(K))}$ for all $f\in J_1$, $g\in J_2$. Furthermore the set $J_1 = \{f\in C(K): supp(f)\subset V_1\}$
is an ideal of $C(K)$. Since $C(K)$ is a regular algebra, $J_1$ coincides with the ideal $J_1^3 $, the linear
span of all products $f_1f_2f_3$, where $f_i\in J_1$ (for each $f\in J_1$ one can find $f_1,f_2\in J_1$ equal
$1$ on $supp(f)$, so $f_1f_2f = f$). On the other hand it is not hard to see that $J_1^3$ is linearly generated
by all functions $f^3$ with $f\in J_1$: it suffices, for each  $f_1,f_2,f_3$, to consider the sum
$$\sum_{i=1}^3(f_1+\omega_if_2 + \omega_i^2f_3)^3$$
where $\omega_i$ are the cubic roots of $1$. This proves our claim.

Let $F(x,y) = \sum_{n=1}^{\infty} a_n(x)b_n(y) \in V(K)$ and $supp(F) \subset K_1\times K_2$ where $K_i$ are
disjoint compacts. We claim that $F\in \overline{\T_{Lie}(C(K))}$. Indeed let $U_i$ $\;(i=1,2)$ be disjoint open
sets containing $K_i$, and $p_i\in C(X)$ be such that ${supp}( p_i) \subset U_i$, $p_i(x) = 1$ for $x\in K_i$.
Then
$$F(x,y) = p_1(x)p_2(y)F(x,y)
= \sum_{n=1}^{\infty} a_n(x)p_1(x)b_n(y)p_2(y).$$ Since $supp(a_n(x)p_1(x))\subset U_1$,
$supp(b_n(x)p_2(x))\subset U_2$, we have, by the above, that $$a_n(x)p_1(x)b_n(y)p_2(y)\in \T_{Lie}(C(K)).$$
Hence $F\in \overline{\T_{Lie}(C(K))}$.

Denote by $\Delta$ the diagonal of $K\times K$: $\Delta = \{(x,x): x\in K\}$. Let $F$ be arbitrary function in
$V(K)$ with $supp F \cap \Delta = \emptyset$. One can choose a finite covering of $supp( F)$ by rectangular open
sets $U_1^n\times U_2^n$ with $\overline{U_1^n\times U_2^n}\cap \Delta = \emptyset$ and $\overline{U_1^n}\cap
\overline{U_2^n} = \emptyset$. Let $\varphi_n$ be the partition of unity corresponding to this covering. Then
each function $F(x,y)\varphi_n(x,y)$ belongs to $\overline{\T_{Lie}(C(K))}$ by the above. Hence $F(x,y) = \sum_n
F(x,y)\varphi_n(x,y) \in \overline{\T_{Lie}(C(K))}$.

Suppose now that $F\in V(K)$ is arbitrary function vanishing on
$\Delta$. Since $\Delta$ is a set of spectral synthesis in $V(K)$
(see \cite{Var}), there is a sequence $(F_n)_{n=1}^{\infty}$ of
elements of $V(K)$ such that $F_n\to F$ (by norm of $V(K)$)  and
$\supp F_n \cap \Delta = \emptyset$. Since all $F_n$ belong to
$\overline{\T_{Lie}(C(K))}$, $F\in \overline{\T_{Lie}(C(K))}$.
\end{proof}

\begin{remark}\label{SD}
 The above proof literally extends to the class (SD) of all regular commutative Banach algebras $B$, such that
the diagonal is the set of spectral synthesis for $B\widehat{\otimes}B$. It was proved in \cite{Sh90} that this
class is quite wide: it contains any regular Banach algebra generated by all its bounded subgroups. In
particular (SD) contains the group algebras  and moreover all regular quotient algebras of measures on locally
convex abelian groups. Thus  we obtain
\begin{corollary} The equality (\ref{banalg}) holds for the group algebras of discrete abelian groups.
\end{corollary}
{\text Note that even for the Wiener-Fourier algebra $WF(\mathbb{T})$ of periodical functions with absolutely
summing Fourier series (the group algebra of $\mathbb{Z}$) the equality (\ref{banalg}) cannot be deduced
directly from purely algebraic results, because $WF(\mathbb{T})$ does not have one generator. It has the dense
subalgebra of trigonometrical polynomials, but this subalgebra does not possess  the property (\ref{mainalg})}.
\end{remark}

\medskip

An example of a Banach algebra for which (\ref{banalg}) fails, can be constructed by modifying one of our
algebraic counterexamples:
\begin{theorem}\label{counter-Ban}
If $B$ is the algebra of absolutely summing Taylor series on
${\mathbb D}^2$ then the equality (\ref{banalg}) is not true.
\end{theorem}
\begin{proof}
An element of $B$ is a function of the form $f(\vec x) =
\sum_{n=0}^{\infty} f^{(n)}$, where $\vec x = (x_1,x_2)$,
$f^{(n)}(\vec x)$ is the uniform component of degree $n$:
$f^{(n)}(\vec x) = \sum_{i+j=n}a_{ij}x_1^ix_2^j$, and
$\sum_{i,j=1}^{\infty}|a_{ij}|<\infty$. Denote by $P$ the projection
onto the space of polynomials of degree $\le 2$: $Pf =
f^{(0)}+f^{(1)} + f^{(2)}$. It is continuous on $B$; as a
consequence the projection $Q = P\otimes P$ is continuous on
$B\widehat{\otimes}B$.

It is clear that the polynomial $p(\vec x, \vec y) = (x_1-y_1)x_2$
(or in tensor form $x_1x_2\otimes 1 - x_2\otimes x_1$) belongs to
${\N_{Lie}(B)}$. We claim that it does not belong to
$\overline{\T_{Lie}(B)}$.

Suppose the contrary: $p = \lim_{k\to \infty} F_k$ with all $F_k$ in
${\T_{Lie}(B)}$. This means that each $F_k$ is a sum of products of
functions of the form $a(\vec x)\otimes 1 - 1\otimes a(\vec x)$. Non
restricting generality we may consider only the case that $a(\vec x)
= x_1^ix_2^j$.

Since $Q$ is continuous, $p = \lim QF_k$. But $QF_k \neq  0$ only
if $F_k$ has a summand of the form
\begin{multline*}\lambda_1(x_1-y_1)^2 + \lambda_2(x_1-y_1)(x_2-y_2) +\\
\lambda_3(x_2-y_2)^2 + \lambda_4(x_1^2-y_1^2)+
\lambda_5(x_1x_2-y_1y_2)
 +\lambda_6(x_2^2-y_2^2) + g_1 + g_3+g_4,\end{multline*} where $g_i$ are
 polynomials of order $i$,
 and in this case $QF_k$ is of this form.
 Hence we obtain that $p $ is a limit of such functions. Since the
 space of these functions is finite-dimensional, $p$ belongs to it.
 Hence \begin{multline*}p = \lambda_1(x_1-y_1)^2 +
\lambda_2(x_1-y_1)(x_2-y_2) +\\ \lambda_3(x_2-y_2)^2 +
\lambda_4(x_1^2-y_1^2)+ \lambda_5(x_1x_2-y_1y_2)
 + \lambda_6(x_2^2-y_2^2).\end{multline*} As the proof of Theorem \ref{counterex}
 shows, this is impossible.
\end{proof}

\mm

\section{Algebras of compact operators}

\medskip

 Let us now consider the algebra $\K(\fX)$ of all compact
operators on a Banach space $\fX$. We assume that $\fX$ has the
approximation property --- the compact operators are norm-limits
of finite-rank operators. The next result shows that the equality
(\ref{baneq}) "almost holds" for $\widetilde{\K(\fX)}$.

\begin{theorem}\label{tensK(X)}
 If $B = \widetilde{\K(\fX)}$ then $\N_{Lie}(B)^2
\subseteq \overline{\T_{Lie}(B)}$.
\end{theorem}
\begin{proof}
Let $x,y\in \N_{Lie}(B)$. By Lemma \ref{semiideals}, $x =
\sum_{i=1}^n d_i x_i$, $y = \sum_{j=1}^m x_j'd_j'$ for some $d_i,
d_j'\in \T_{Lie}(B)$, $x_i, x_j'\in B\otimes B^{op}$. Hence $xy =
\sum_{i, j} d_ix_ix_j'd_j' \in \T_{Lie}(B) (B\otimes B^{op})
\T_{Lie}(B)$ and we get
\begin{equation}\label{part}
\N_{Lie}(B)^2\subseteq \T_{Lie}(B) (B\otimes B^{op}) \T_{Lie}(B).
\end{equation}

Let $a, b,  y, z\in B$, $a_n,b_n,y_n,z_n \in \widetilde{F(\fX)}$
with $a_n\to a$, $b_n\to b$, $y_n\to y$, $z_n\to z$, where as usual
$F(\fX)$ is the algebra of all finite-rank operators on $\fX$.  Then
$$(a\ot 1-1\ot a)(y\ot z)(b\ot 1 - 1\ot b) = \lim (a_n\ot 1-1\ot
a_n)(y_n\ot z_n)(b_n\ot 1 - 1\ot b_n)$$ and $(a_n\ot 1-1\ot
a_n)(y_n\ot z_n)(b_n\ot 1 - 1\ot b_n)\in
\T_{Lie}(\widetilde{F(\fX)}) \widetilde{F(\fX)}\otimes
{\widetilde{F(\fX)}}^{op}
 \T_{Lie}(\widetilde{F(\fX)})$. By Lemma
\ref{semiideals} $\T_{Lie}(\widetilde{F(\fX)})
\widetilde{F(\fX)}\otimes {\widetilde{F(\fX)}}^{op}
\T_{Lie}(\widetilde{F(\fX)})\subseteq \N_{Lie}(\widetilde{F(\fX)})$
which in its turn coincides with $\T_{Lie}(\widetilde{F(\fX)})$ by
Theorem \ref{Ftens}. Thus $(a\ot 1-1\ot a)(y\ot z)(b\ot 1 - 1\ot
b)\in \overline{\T_{Lie}(\widetilde{F(\fX)})} \subseteq
\overline{\T_{Lie}(B)}$, for any $a, b,  y, z\in B$, whence
$\T_{Lie}(B)B\otimes B^{op} \T_{Lie}(B) \subseteq
\overline{\T_{Lie}(B)}$. By (\ref{part}) $$\N_{Lie}(B)^2 \subseteq
\overline{\T_{Lie}(B)}.$$
\end{proof}

\medskip


Our next goal is to establish the equality
\begin{equation}\label{banel-1}
\overline{\D_{Lie}(B)} = \overline{\M_{Lie}(B)}
\end{equation}
 for $B = \K(H)$, the
algebra of all compact operators on a Hilbert space $H$.

It follows easily from the Lomonosov's Theorem (see \cite{Lom})
that in a reflexive Banach space $\fX$ with the approximation
property each transitive algebra of compact operators is norm
dense in the algebra $\K(\fX)$ of all compact operators. We need
the following extension of this result.

\begin{lemma}\label{comp}  Let $\fX$ be a Banach space with the approximation
property, $W\subset \K(\fX)$ a closed algebra without (closed)
invariant subspaces, $Y$ a closed complemented subspace of $ \fX^*$.
Suppose that the following conditions are fulfilled:

(i) $W^*\fX^*\subset Y$,

 (ii) There is no proper non-zero closed
 subspace of $Y$  invariant for $W^*$.

 Then $W = \{T\in \K(\fX): T^*\fX^*\subset Y\}$.
\end{lemma}
\begin{proof} Prove first of all that $W$ contains a rank
one operator with nonzero trace. By Lomonosov's Lemma (see
\cite{Lom}) there is such operator $T\in W$ that $\sigma(T) \neq
\{0\}$. Let $0\neq \lambda \in \sigma(T)$ and $P$ be the
corresponding spectral projection. Then $P$ is finite-dimensional
and belongs to $W$. Let $W_0 = PWP|_{P\fX}$. Since this algebra
has no invariant subspaces it coincides with $B(P\fX)$ by
Burnside's Theorem. Hence it contains rank one operator $T$ such
that $tr T \neq 0$. Since we can identify $W_0$ with a subalgebra
of $W$ we have $T\in W$. Writing $T = x_0\otimes f_0$ where
$x_0\in \fX$, $f_0\in \fX^*$, we have that $0\neq tr T =
f_0(x_0)$. Since $Im\; T^* = \mathbb C f_0$ we have $f_0\in Y$.

Consider the set $\{x\in \fX: x\otimes f_0 \in W\}\subset \fX$. It
is a closed
 invariant subspace for $W$ and hence it coincides with $\fX$.
 Similarly the set $\{f\in \fX^*: x_0\otimes f \in W\}\subset \fX^*$ is closed
 invariant for $W^*$ subspace of $Y$ and hence it coincides with $Y$.
 So for any $x\in \fX$, $f\in Y$ we have $x\otimes f_0\in W$,
 $x_0\otimes f\in W$ whence $x\otimes f = \frac{1}{f_0(x_0)}
 (x\otimes f_0)(x_0\otimes f) \in W$.  Hence $W$ contains any
finite rank operator $T$ such that $T^*\fX^*\subset Y$. If we
denote by $W_1$ the algebra of all compact operators $T$ such that
$T^*\fX^*\subset Y$, then it can be said that $W_1\cap
\mathcal{F}(H) \subset W$. So to show that $W = W_1$ we have only
to establish that $W_1\cap \mathcal{F}(H)$ is norm dense in $W_1$.

 Let $P:\fX^*\to \fX^*$ be a projection on $Y$ (it exists because $Y
\subset \fX^*$ is  assumed to be complemented). Denote by $j:
\fX\to \fX^{**}$ the canonical  inclusion. We claim that for any
compact operator $T\in \K(\fX)$, the subspace $j(\fX)$ of
$\fX^{**}$ is invariant for the operator $T^{**}P^*$. Indeed for
any finite rank operator $T = \sum_{i=1}^N x_i\otimes f_i$, we
have $T^* = \sum_{i=1}^N
 f_i \otimes j(x_i)$ and for any $z\in \fX$, $g\in \fX^*$
 \begin{multline*}\left(T^{**}P^*j(z)\right)(g) = (P^*j(z))(T^*g) =
(P^*j(z))\left(\sum_{i=1}^N g(x_i)f_i\right) =\\ \sum_{i=1}^N
g(x_i)(Pf_i)(z) =
 j\left(\sum_{i=1}^N(Pf_i)(z)x_i\right)(g)\end{multline*}
   whence $$T^{**}P^*j(z)=
 j\left(\sum_{i=1}^N(Pf_i)(z)x_i\right)$$ so that $j(\fX)$ is invariant for
 $T^{**}P^*$. Let   $T = lim_{n\to \infty} T_n$, where $T_n$ are of finite
 rank. Then $\|T^{**}P^* - T_n^{**}P^*\| \le \|T^{**} - T_n^{**}\|
 \to 0$ and since the closed subspace $j(\fX)$ is invariant for all
 $T_n^{**}P^*$ we get that it is invariant for $T^{**}P^*$.

 Now we can define for each $T\in \K(\fX)$, an operator $\hat{T}\in \K(\fX)$ by the equality
\begin{equation}\label{intertw}
 j\hat{T} = T^{**}P^*j.
\end{equation}
It is easy to see from (\ref{intertw}) that the map $T\to \hat{T}$
is linear and continuous: $\|\hat{T}\|\le \|P\|\|T\|$.

If $T=x\otimes f$ then $\hat{T} = x\otimes Pf$ while
$(\hat{T})^*(Y^*) \subset \mathbb{C}Pf \subset Y$. Thus
$\hat{T}\in W_1$ for any rank one operator $T$. By linearity and
continuity of the map $T\to \hat{T}$ we conclude that this is true
for all $T\in \K(\fX)$.

\noindent Now let us show that $\hat T = T$ for each $T\in W_1$.
For any $x\in \fX$, $g\in X^*$, we have $g(\hat Tx) =
(T^{**}P^*j(x))(g) = j(x)(PT^*g) = (PT^*g)(x)$ and since $T^*g\in
Y$ we have $PT^*g = T^* g$ whence $g(\hat T x) = (T^*g)(x) =
g(Tx)$. Thus we get $\hat T = T$.

Now we can finish the proof. Let $K\in W_1$. Since $\fX$ has the
approximation property, one can choose finite rank operators $K_n$
such that $K_n\to K$.  Since  the map $T\mapsto \hat T$ is
continuous we have $\hat K_n\to \hat K$. Then $K = \hat K = \lim
\hat K_n$. But $\hat K_n\in W_1\cap \mathcal{F}(\fX) \subset W$.
Since $W$ is closed we get $K\in W$.

\end{proof}

\m

Let $B = \K(H)$, the algebra of all compact operators on a Hilbert
space $H$. We will apply Lemma \ref{comp} in the case $\fX = B$. The
dual Banach space $B^*$ of $B$ is $C_1(H)$, the ideal of all trace
class operators. We denote by $C_1^0(H)$ the subspace in $C_1(H)$
consisting of all operators with zero trace.

\begin{corollary}
$\overline{\D_{Lie}(B)}\bigcap \K(B) = \{T\in \K(B): T^*B^*\subset
C_1^0(H)\}$.
\end{corollary}
\begin{proof} Note first of all that $B = \K(H)$ has the approximation property.
Indeed if $P_n$ is a sequence of projections increasing to $1_H$,
then the compact maps $T\to P_nTP_n$ strongly tend to $1_{\K(H)}$.

Set $W = \overline{\D_{Lie}(B)}\bigcap \K(B)$, $Y = C_1^0(H)$.
Obviously $Y$ has codimension 1 in $C_1(H)$, so it is complemented.
For any $a\in \A, y\in C_1(H) , x\in \B$, we have
$$\left((L_a-R_a)^*(y)\right)(x) = y((L_a-R_a)(x)) = y([a,x]) = \text{tr }y[a, x] =
- \text{tr }[a, y]x = ([y, a])(x),$$ so $(L_a-R_a)^*(y)= [y, a]$ and
we see that $(L_a-R_a)^*$ maps $C_1(H)$ to $Y$. It follows that all
operators in $\overline{\D_{Lie}}(B)^*$ map $C_1(H)$ to $Y$. Hence
$W^*C_1(H)\s Y$. We proved that the condition (i) of Lemma
\ref{comp} is fulfilled.

To establish the condition (ii) let us prove firstly that for each
pair $p,q$ of finite rank projections  with $pq = 0$, the operator
$L_pR_q$ belongs to $W$. It is clear that this operator is compact
so we have to show only that it belongs to $\D_{Lie}(B)$. The
operator $T = L_pR_q+L_qR_p = -(L_p-R_p)(L_q-R_q)$ belongs to
$\D_{Lie}(B)$, hence $S = (L_p-R_p)T = L_pR_q - L_qR_p$ also belongs
to $\D_{Lie}(B)$. It follows that $L_pR_q = (T+S)/2\in \D_{Lie}(B)$.

Let $Y_0\s Y$ be a closed subspace invariant for $W^*$. Denote by
$Z$ the annihilator of $Y_0$ in $C_1(H)^* = B(H)$. This subspace is
invariant with respect to all operators $T^{**}: T\in W$. It is not
difficult to see that the second adjoint of a multiplication
operator on $K(H)$ is the "same" multiplication operator on $B(H)$.
Thus $pZq \s Z$ if $p,q$ are finite-rank projections with $pq = 0$.

Suppose now that $q$ is an arbitrary projection satisfying the
condition $pq = 0$. Then there is a sequence  of finite-rank
projections $q_n \le q$ which tends to $q$ in strong operator
topology. Since $pZq_n \s Z$ and $Z$ is *-weakly closed we obtain
that $pZq = 0$. Dealing in the same way with $p$ we conclude that
$pZq$ for any projections $p,q$ satisfying the condition $pq = 0$.
In particular $pz(1-p)\in Z$ and $(1-p)zp\in Z$ for each projection
$p$ and each $z\in Z$. Subtracting we get that $pz-zp\in Z$.  Using
the fact that $B(H)$ is linearly generated by projections (or just
Spectral Theorem and the closeness of $Z$) we obtain that $az-za\in
Z$ for all $a\in B(H)$ and $z\in Z$. Thus  $Z$ is a Lie ideal of
$B(H)$. Since $Z$ is *-weakly closed and non-zero we have, by
\cite{Mi} (see also \cite{Mu}, \cite{FM}), that $Z = B(H)$ or $Z =
\mathbb{C}1$. Thus $Y_0 = 0$ or $Y_0 = Y$.

 We proved that the condition (ii) of Lemma \ref{comp} is fulfilled. Now applying Lemma
 \ref{comp} we get
 that $\overline{\D_{Lie}(B)}\bigcap \K(B) = \{T\in \K(\B): T^*C_1(H)\subset C_1^0(H)\}$.
 \end{proof}

 \medskip

 \begin{lemma}\label{root} Let $H$ be a Hilbert space, $A, B \in \K(H)$.
 Then there exist $K, X, Y\in \K(H)$ such that
$A+K = X^2$ and $B-K = Y^2$.
\end{lemma}
\begin{proof} Write $A+B$ in the form $M+iN$ where $M,N$ are
hermitian. Then both components are normal and compact hence $M =
X^2$, $N = Y^2$ with compact $X,Y$. So $A+B = X^2 + Y^2$. Set $K = B
- Y^2$. Then $A+K = X^2$, $B-K = Y^2$.
\end{proof}

\mm

\begin{theorem} $\overline{\D_{Lie}(\K(H))} = \overline{\M_{Lie}(\K(H))}$.
\end{theorem}
\begin{proof} Denote $\K(H) = \A$. Let $T = \sum L_{a_i}R_{b_i}\in \M_{Lie}(\A)$, where
$a_i, b_i \in \A$. This is a compact operator on $\A$. For any $u\in
C_1^0(H)$ and $x\in \A$ we have $$(T^*u)(x) = u(Tx) = \text{tr }uTx
= \text{tr
 } u\sum a_ixb_i = \text{tr } \sum b_iua_ix = (\sum b_iua_i)(x)$$
 (we use the same notation for a nuclear operator
 and the corresponding functional on the space of operators).
 So $T^*u = \sum b_iua_i$ and $\text{tr }T^*u =  \text{tr
}\sum a_ib_i u = 0$. Hence $T^*\A^*\subset C_1^0(H)$. By Lemma
\ref{comp}, $T\in \overline{\D_{Lie}(\A)}$.

Note that $\M_{Lie}(\A)$ is generated by elements of the form  $T
= \sum L_{a_i}R_{b_i} + R_c + L_d$, where $a_i, b_i, c, d \in \A$,
$\sum_ia_ib_i + c + d = \sum_i b_ia_i + c + d = 0$ . So it
suffices to prove that any such operator belongs to
$\overline{\D_{Lie}(\A)}$. By Lemma \ref{root}, there exist
elements $k, x, y\in \A$ with $c-k = x^2$, $d+k = y^2$. Then $T =
\sum L_{a_i}R_{b_i} + R_{x^2} + L_{y^2}$. Let us consider an
operator $S = \sum L_{a_i}R_{b_i} + L_xR_x + L_yR_y$. Then $S$ is
compact and belongs to $\M_{Lie}(\A)$ because $\sum_ia_ib_i + x^2
+ y^2 = \sum_ib_ia_i+x^2+y^2 =0$. By what we proved above, $S\in
\overline{\D_{Lie}(\A)}$. Let us show that $T-S\in
\overline{\D_{Lie}(\A)}$. For this we should prove that $R_{x^2} -
L_xR_x\in \overline{\D_{Lie}(\A)}$ and $L_{y^2} - L_yR_y\in
\overline{\D_{Lie}(\A)}$.

Consider the homomorphism $\phi$ from the free algebra $\mathcal F$
with one generator $a$ to $\A$, which sends $a$ to $x$.
Correspondingly we have a homomorphism $\gamma$ from $\mathcal
F\otimes \mathcal F$ to $El(\A)$: $\gamma(u{\otimes}v) =
L_{\phi(u)}R_{\phi(v)}$. Clearly $\gamma(\N_{Lie}(\mathcal F))
\subset \M_{\text{\emph{Lie}}}(\A)$, $\gamma(\T_{Lie}(\mathcal F))
\subset \D_{\text{\emph{Lie}}}(\A)$, $\gamma(1\otimes a^2 - a\otimes
a) = R_{x^2} - L_xR_x$. By Theorem \ref{polyn}, $\N_{Lie}(\mathcal
F) = \T_{\text{\emph{Lie}}}(\mathcal F)$; since $1\otimes a^2 -
a\otimes a \in \N_{Lie}(\mathcal F)$ we get $R_{x^2} - L_xR_x\in
\overline{\D_{Lie}(\A)}$. Using the same arguments we obtain
$L_{y^2} - L_yR_y\in \overline{\D_{Lie}(\A)}$.

So $S$ and $T-S$ are in $\overline{\D_{Lie}(\A)}$ whence $T\in
\overline{\D_{Lie}(\A)}$.
\end{proof}

\section{Applications to Lie ideals}

Recall that a $C^*$-algebra $F$ is called uniformly hyperfinite
(UHF) if it is the closure of the union of an increasing sequence
of $C^*$-subalgebras $F_n$ isomorphic to full matrix algebras.

Let us denote by $F_{\infty}$ the union of the algebras $F_n$.

It is known that $F$ has a unique normalized trace $\tau$. We set
$F_{\tau} = \Ker \tau$.

Let $\I$ be the identity operator on $F$. We will denote by
$\overline{\D_{Lie}(F)+\mathbb{C}\I}^s$ the closure
 of the unitalization of the algebra $\D_{Lie}(F)$ in the strong operator
 topology.
\begin{lemma}\label{UHF}
Let $F$ be a uniformly hyperfinite algebra. Then there is a sequence
$P_n$ of finite rank norm-one
 projections in  $\overline{\D_{Lie}(F)+\mathbb{C}\I}^s$
  which strongly tends to  $\I$.
\end{lemma}
\begin{proof}
It is well known that $F$ can be realized as  the infinite C*-tensor
product $M_1\ot M_2\ot ...$ of matrix algebras $M_i =
M(n_i,\mathbb{C})$. We denote by $F_n$ the product of the first $n$
factors. Then $F_n'$, the commutant of $F_n$ in $F$, is the product
of all factors from $n+1$ to $\infty$.

Let $K_n$ be the operator
\begin{equation}\label{amen}
 x\to
\int_{U(F_n)}uxu^*du
\end{equation}
 where $U(F_n)$ is the unitary
group of $F_n$.

\mm

 Claim 1. $\|K_nx - \tau(x)1\| \to 0$ for each $x\in F$.

Indeed since $\|K_n\| = 1$, it suffices to check this for $x\in
F_m$. But for $n=m$, it is well known that $K_n(x) = \tau(x)1$.
Hence the same equality holds for $n
> m$.

\mm

Let us denote $\overline{\D_{Lie}(F)}^s$ by $\E$, for brevity.

 Claim 2. Let $P_0$ be the operator $x\to
\tau(x)1$. Then $P_0 - \I \in \E$.

Indeed it suffices to show that $K_n - \I \in \E$. Let $\varepsilon
> 0$. Let $W_1,...,W_N$ be measurable subsets of $U_n$ with $m(W_k)
= 1/N$ and $diam(W_k) \le \varepsilon$. Choose $u_k\in W_k$ and set
$T_n(x) = 1/N\;\sum_ku_kxu_k^*$. Then $\|T_n - K_n\|\le
2\varepsilon$. On the other hand $T_n - \I = 1/N
\sum_k(L_{u_k}R_{u_k^*} - L_1R_1) \in \E$. Indeed all operators
$L_{u_k}R_{u_k^*} - L_1R_1$ clearly belong to $\M_{Lie}(F_n)$. Since
$F_n$ is a full matrix algebra they belong to $\D_{Lie}(F_n) \s
\D_{Lie}(F) \s \E$, by Theorem \ref{tanya}.

Since $\varepsilon$ is arbitrary, $K_n-\I\in \E$.

\mm

Let us denote by $P_k$ the expectation onto the subalgebra $F_k$. It
can be calculated as the limit of operators $K_{n,k}$, which are
defined by the formula, similar to (\ref{amen}) but with integration
by the unitary group of the algebra $M_{k+1}\ot M_{k+2}\ot...\ot
M_{k+n}\s F_k'$. The arguments similar to the above show that $P_k -
\I\in\E$.

Since $P_k$ are projections onto $F_k$ and $\|P_k\| = 1$, we
conclude that $P_k\to \I$ in the strong operator topology.
\end{proof}

\mm

 The lemma implies a localization result for Lie
ideals in the projective tensor products with uniformly hyperfinite
algebras:

\begin{corollary}\label{calUHF}
Let $B$ be an arbitrary unital Banach algebra, $A = B\hat{\ot} F$,
where $F$ is a UHF algebra. Then each closed Lie ideal $L$ in $A$ is
the closure of $L_{\infty}: = L\cap (B\ot F_{\infty})$.
\end{corollary}
\begin{proof}
Let $a= \sum_jb_j\ot x_j\in L$. Then
\begin{equation}\label{twoterms}
a = \sum_jb_j\ot P_kx_j + \sum_jb_j\ot (x_j-P_kx_j)
\end{equation}
where $P_k$ are the projections constructed in Lemma \ref{UHF}.
Clearly $1\otimes \overline{\D_{Lie}(F)}^s \s
\overline{\D_{Lie}(B\hat{\otimes}F)}^s$ and this implies that
operators in $1\otimes \overline{\D_{Lie}(F)}^s$ preserve Lie ideals
of the algebra $B\hat{\otimes}F$. Since, by Lemma \ref{UHF}, $1-P_k
\in \overline{\D_{Lie}(F)}^s$, the second term in (\ref{twoterms})
belongs to $L$. Hence the first one belongs to $L$. Moreover it
belongs to the tensor product of $B$ and $F_k$ which is contained in
$B\otimes F_{\infty}$. Therefore it belongs to $L_{\infty}$. The
second term tends to $0$ when $k\to \infty$. Hence the first one
tends to $a$. We obtain that $L_{\infty}$ is dense in $L$.
\end{proof}

\medskip

Applying Theorem 4.14 of \cite{BKS} we obtain the description of
closed Lie ideals of $B\hat{\otimes}F$ in terms of Lie ideals of
$B$.

\begin{corollary}
For each closed Lie ideal $L$ of $A = B\hat{\ot}F$, there is a
closed ideal $I$ of $B$ and a closed Lie ideal $M$ of $B$ such that
$L = I\hat{\otimes}F_{\tau} + M\hat{\otimes} 1$.
\end{corollary}

\mm

\medskip

\section{Applications to invariant subspaces of operator Lie algebras}

A well known result of Arveson \cite{Arv} states that if an algebra
of  operators in a Hilbert space contains a maximal abelian
selfadjoint algebra (masa, for short) then either it has a
non-trivial invariant subspace or it is dense in $B(H)$ with respect
to the ultra-weak topology. We now extend this result to Lie
algebras.

Let $\D_i$ be masas in $\B(H_i)$, $i = 1,2$. All Hilbert spaces will
be assumed separable so $\D_i$ can be realized in coordinate way:
$\D_1 = L^{\infty}(X,\mu)$, $\D_2 = L^{\infty}(Y,\nu)$ acting on
$H_1 = L^2(X,\mu)$ and respectively $H_2 = L^2(Y,\nu)$ by
multiplications. Here $X$,$Y$ are metrizable locally compact spaces,
$\mu$, $\nu$ are regular measures.

 We identify a subset $K$ of $X$
with a projection in $\D_1$ (multiplication by $\chi_K$) and with
the range of this projection (the space of all functions in $
L^2(X,\mu)$, vanishing almost everywhere outside  $K$).

A set $\kappa\s X\times Y$ is called marginally null (m.n.) if it is
contained in $(P\times Y) \cup (X\times Q)$, where $P$ and $Q$ have
zero measure. For brevity, we write $\kappa = 0$ ($\kappa\neq 0$) if
$\kappa$ is (respectively is not) marginally null.

A set is called $\omega$-open if up to a m.n. set, it coincides with
the union of a countable family of "rectangulars"$\;$ $P\times Q$.
The complements to $\omega$-open sets are called $\omega$-closed.

We have to define the projections of a subset in $X\times Y$ to the
components. Let firstly for any family  $\cP = P_{\lambda}$ of
measurable subsets of $X$, define its supremum and infimum. Namely,
$\vee(\cP)$ is the subset of $X$ that corresponds to the closed
linear span of all subspaces $P_{\lambda}H$, while $\wedge(\cP)$
corresponds to  their intersection. In other words we take infimum
and supremum in the measure algebra of $(X,\mu)$ (or in the lattice
of projections of $L^{\infty}(X,\mu)$).

We set now  $$\pi_1(\kappa) = X \setminus \vee\{P: (P\times Y)\cap \kappa \text{ is m. n.}\},$$
$$\pi_2(\kappa) = Y \setminus \vee\{P: (X\times P)\cap \kappa \text{ is m. n.}\}.$$

Let us call a rectangular $P\times Q$ {\it non-essential} for a
family $\cL\s B(H_1,H_2)$, if
\begin{equation}\label{lishnie}
Q\cL P =0.
\end{equation}

We say that a set $\kappa\s X\times Y$ supports $\cL$, if any
rectangular, non-intersecting  $\kappa$, is non-essential for $\cL$.
It is known that among all $\omega$-closed sets supporting  $\cL$,
there is the smallest one (up to a m.n. set it is contained in each
supporting $\cL$ set). It is called the {\it support} of $\cL$ and
denoted $\supp \cL$.

For any  $\omega$-closed set $\kappa$, we denote by
$\M_{max}(\kappa)$ the set of all operators supported by $\kappa$.
It is well known (see for example \cite{ST}) that it is a
$\D$-bimodule and $\supp(\M_{max}(\kappa)) = \kappa$.

In what follows $X = Y$, $\mu = \nu$, $\D_1 =\D_2 = \D$, $H_1 = H_2
= H$.

\begin{lemma}\label{proj}
Let $\E$ be an uw-closed $\D$-bimodule with $\supp (\E) = \kappa$.
Then

(i) $\overline{\E H} = \pi_2(\kappa)$;

(ii) $\ker \E = X\setminus \pi_1(\kappa)$.
\end{lemma}
\begin{proof}
(i) Since $\E$ is an uw-closed $\D$-bimodule, $\overline{\E H}$ is closed under multiplication by functions from
$L^{\infty}(X,\mu)$ and hence, as is well known, consists of all functions from $L^2(X,\mu)$ vanishing on some
subset of $X$. Thus $\overline{\E H}\subset X$.

Let $P\bigcap \overline{\E H} \neq \emptyset$. Then $\overline{P\E H}\neq \emptyset$, that is $P\E\neq
\emptyset$, or, equivalently, $(X\times P)\bigcap \kappa\neq \emptyset$, that means $P\bigcap \pi_2(\kappa)\neq
\emptyset$. Thus $\overline{\E H} \subset \pi_2(\kappa)$.

 Let $P\bigcap \pi_2(\kappa)= \emptyset$. It is
equivalent to $(X\times P)\bigcap \kappa= \emptyset$, whence $P\E =0$ and $P\bigcap \E H = \emptyset$. Thus
$\pi_2(\kappa) \subset \E H $ and hence $\pi_2(\kappa) = \overline{\E H} $.

(ii) Let $\tilde \kappa$ be the support of $\E^{\ast}$. It is easy to see that $\tilde \kappa = \{(y, x)\;|\;
(x, y)\in \kappa\}$. Hence we have $\ker \E = (\overline{\E^{\ast}H})^{\bot} = (\pi_2(\tilde \kappa))^{\bot} =
(\pi_1(\kappa))^{\bot} = X \setminus \pi_1(\kappa)$.
\end{proof}

\begin{lemma}\label{inter}
If $\E$ is a family of operators, $P, Q$ --- projections in $ \D$ then

(i) ${\supp} (Q\E) = (X\times Q)\cap {\supp}(E)$,

(ii) ${\supp} (\E P) = (P\times Y)\cap {\supp} (\E)$.
\end{lemma}
\begin{proof} Since both sides of (i) are $\omega$-closed,
 it suffices to prove that for all projections $R, S \in \D$,
the condition $(R\times S) \bigcap {\supp} (Q\E) = 0$ holds if and
only if $(R\times S) \bigcap (X\times Q)\cap {\supp}(E) = 0$. It is
easy to see that the both conditions are equivalent to $SQ\E R=0$.

\noindent The proof of (ii) is similar.
\end{proof}

\medskip

Let us say that a set $\kappa\s X\times X$ is {\it the graph of an order} if for each rectangular $P\times Q$
non-intersecting $\kappa$,
\begin{equation}\label{order}
\pi_2(\kappa\cap (P\times X)) \wedge \pi_1(\kappa\cap (X\times Q)) = \emptyset.
\end{equation}

After this preliminary work we turn to the  consideration of
uw-closed Lie subalgebras of $B(H)$ that contain masa.

\begin{proposition}\label{bim}
If an uw-closed Lie subalgebra $\cL\s B(H)$ contains a masa $\D$
then it is a $\D$-bimodule.
\end{proposition}
\begin{proof}
 Let us prove firstly that if $A,B\in \D$  and $AB = 0$ then $A\cL B\s \cL$.

Indeed let us choose a dense separable subalgebra $\D_0$ of $D$
which contains $A,B$. Then we may assume that $X$ is the character
space of $\D_0$ and $\D_0 = C(X)$. Now $\cL$ is a Lie submodule of a
$\D_0$-bimodule; by Theorem \ref{Lie}, it is stable under
multiplying on elements of $\D_0\widehat{\otimes}\D_0$ which vanish
on the diagonal. The product $A\otimes B$ is such an element. Hence
$A\cL B\s \cL$.

Let now $\cP$ be a decomposition $X = P_1\cup P_2\cup ... \cup P_n$
of $X$. Set $\E_{\cP}(T) = \sum_{i=1}^nP_iTP_i$, for any $T\in
B(H)$. Then, for each $T\in \cL$,  $T - \E_{\cP}(T):= \F_{\cP}(T) =
\sum_{i\neq j}P_iTP_j \in \cL$,  by the above (because $P_iP_j =
0$). If $A\in \D$ then $ AT = A\E_{\cP}(T) + A\F_{\cP}(T)$. Again
$A\F_{\cP}(T)\in \cL$ because $(AP_i)P_j = 0$. Thus $AT -
A\E_{\cP}(T) \in \cL$, for any decomposition $\cP$. Taking a
decreasing sequence $\cP_n$ of the decompositions, we may say that
$AT - S\in \cL$, for any limit point $S$ of operators
$A\E_{\cP}(T)$.

It is easy to see that $S\in \D$. Indeed $A\E_{\cP}(T) = \E_{\cP}(AT)$ commutes with all projections $P_i\in
\cP$, so it commutes with the algebra $\D(\cP)$, generated by $\cP$. Since the union of the algebras
$\D(\cP_n)$, for an appropriate  decreasing sequence $\cP_n$ of decomposition, generates $D$, each limit point
of the sequence $A\E_{\cP_n}(T)$ commutes with $D$ and hence belongs to $D$.

Since $\D\s\cL$, we get that  $AT = (AT-S) + S \in \cL$; similarly
$TA\in \cL$.
\end{proof}

\begin{theorem}\label{graph}
The following conditions on an $\omega$-closed set $\kappa\s X$ are
equivalent:

(i) $\M_{max}(\kappa)$ is a unital algebra;

(ii) $\kappa$ is the support of an algebra that contains $D$;

(iii) $\kappa$ is the support of a Lie algebra that contains $D$;

(iv) $\kappa$ is the graph of an order and contains the diagonal
$\Delta = \{(x,x): x\in X\}$ of $ X\times X$.
\end{theorem}
\begin{proof} We will denote $\M_{max}(\kappa)$ by $\M$, for
brevity.

$(i) \Rightarrow (ii)$. Since $\M_{max}(\kappa)$ is a $\D$-bimodule
 we have that $\D\s \M_{max}(\kappa)$ if (i) holds.

$(ii) \Rightarrow (iii)$ is evident.

$(iii)\Rightarrow (iv)$. Let $\cL$ be a Lie algebra, $\D\s \cL$ and
$\supp(\cL) = \kappa$. Non restricting generality we may assume that
$\cL$ is uw-closed. Hence, by Proposition \ref{bim}, it is a
$\D$-bimodule.

It follows from the inclusion $\D\s \cL$ that $\supp D \s \supp \cL$
that is $\Delta\s \kappa$.

Let $\kappa\cap (P\times Q) = 0$. Since $\Delta\s\kappa$, $P\cap Q =
0$. Then $Q\cL\cL P = [Q\cL, \cL P] \s \cL $, hence $Q\cL\cL P\s
Q\cL P = 0$. It follows that $\cL PH\s \ker(Q\cL)$. By Lemma
\ref{proj}, $\pi_2(\supp(\cL P)) \cap \pi_1(\supp(Q\cL)) = 0$. This
exactly means (if one takes into account Lemma \ref{inter}) that
(\ref{order}) holds. By Lemma \ref{graph}, $\kappa$ is the graph of
an order.

(iv)$\Rightarrow$ (i). If $P\times Q$ does not intersect $\kappa$ then (\ref{order}) holds. Set
$P_1=\pi_1(\kappa\cap(X\times Q))$, $Q_1 = \pi_2(\kappa\cup (P\times X))$. By Lemma \ref{inter}, ${\supp}( \M P)
= (P\times X)\cap \kappa$, ${\supp} Q\M = (X\times Q) \cap \kappa$. By Lemma \ref{proj}, $\overline{\M PH} =
\pi_2(\supp \M P) = \pi_2(\kappa\cup (P\times X))=Q_1$, $\ker Q\M = X\setminus \pi_1(\supp Q\M) =
\pi_1(\kappa\cap(X\times Q)) = X\setminus P_1$. But by (\ref{order}), $Q_1 \s X\setminus P_1$. This means that
$Q\M\M P = 0$, that is $P\times Q$ is non-essential for $\M^2$. Hence ${\supp}(\M^2)\s \kappa$, $\M^2\s \M$.

We proved that $\M$ is an algebra; since $\Delta\s \kappa$ it is
unital.
\end{proof}

\medskip

Let us say, for brevity, that a subspace $\M$ of $B(H)$ is {\it
irreducible} if it has no non-trivial closed invariant subspaces.
Furthermore $\M$ is {\it transitive} if $\overline{\M x} = H$ for
each non-zero vector $x\in H$. It is easy to see that if $M$ is a
unital algebra then these conditions are equivalent.

It was proved by Arveson \cite{Arv} that a transitive bimodule over
a masa is uw-dense in $B(H)$ (for non-separable $H$ it was proved in
\cite{Sh}). As a consequence each irreducible operator algebra
containing a masa is uw-dense in $B(H)$ (the density in the weak
operator topology was previously established in \cite{Arv1}). Now we
extend this result to Lie algebras.

\begin{corollary}
A Lie algebra $\cL$ of operators, containing a masa $\D$, either has
a non-trivial invariant subspace or is $uw$-dense in $B(H)$.
\end{corollary}
\begin{proof}
Suppose that $\cL$ is irreducible. Let $\kappa = \supp(\cL)$ and $\A
= \M_{max}(\kappa)$. By Theorem \ref{graph} (the equivalence (i)
$\Leftrightarrow$ (iii)), $\A$ is an algebra. Since $\cL$ has no
invariant subspaces, so does $\A$. By \cite{Arv1},  $\A = B(H)$ and,
consequently,  $\kappa = X\times X$.

Non-restricting generality we may assume that $\cL$ is uw-closed. By
Proposition \ref{bim}, $\cL$ is a masa-bimodule. Let us prove that
it is transitive. Suppose that $\cL x$ is not dense in $H$, for some
$x\in H$. Let $P$ be the projection on the subspace $(\cL
x)^{\bot}$, $Q$ be the projection on $\D x$. Then $P\cL Q=0$ and,
since the subspaces  $(\cL x)^{\bot}$ and $\D x$ are invariant for
$\D$, the projections $P, Q$ belongs to $\D$. This is a
contradiction with the equality $\kappa = X\times X$. Thus $\cL$ is
a transitive masa-bimodule. Applying \cite{Arv} we conclude that
$\cL$ is $uw$-dense in $B(H)$, $\cL = B(H)$.
\end{proof}


\begin{thebibliography}{99}

\bibitem{Arv1} W.Arveson, A density theorem for operator algebras,
Duke Math.J. \textbf{34} (1967), 635-647.

\bibitem{Arv} W.Arveson, Operator algebras and invariant subspaces,
Annals of Math., \textbf{100} (1974), 433-532.

\bibitem{Ando}  T. Ando, On a pair of commutative contractions, Acta
Sci. Math. \textbf{24} (1963), 88-90.

\bibitem{BKS} M.Bresar, E.Kissin, V.S.Shulman, Lie ideals: from pure
algebra to C*-algebras, J. Reine und Angew. Math., to appear

\bibitem{BMM} K. I. Beidar, W. S. Martindale, A. V. Mikhalev, Rings with
generalized identities, Marcel Dekker, Inc., 1996.

\bibitem{EKS}  Erdos, J. A.; Katavolos, A.; Shulman, V. S., Rank one subspaces
of bimodules over maximal abelian selfadjoint algebras, J. Funct.
Anal. \textbf{157} (1998), no. 2, 554--587.

\bibitem{FM}  C.\thinspace K. Fong and G.\thinspace J. Murphy, Ideals
and Lie ideals of operators, Acta Sci. Math. \textbf{51}(1987),
441-456.

\bibitem{Hers}  I.\thinspace N. Herstein, On the Lie and Jordan rings of a
simple associative ring, Amer. J. Math. \textbf{77} (1955), 279-285.

\bibitem{JR}  N. Jacobson and C. Rickart, Jordan homomorphisms of rings,
Trans. Amer. Math. Soc. \textbf{69}(1950), 479-502.

\bibitem{Lom} V. Lomonosov, Invariant subspaces for operators commuting
with compact operators, Funct. Anal. Appl.\textbf{ 7} (1973)
213-214.



\bibitem{MM1}  W.\thinspace S. Martindale 3rd and C.\thinspace R. Miers,
Herstein's Lie theory revisited, J. Algebra \textbf{98} (1986),
14-37.

\bibitem{Mi}  C.\thinspace R. Miers, Closed Lie ideals in operator
algebras, Canad. J. Math. \textbf{33} (1981), 1271-1278.

\bibitem{Mu}  G.J. Murphy, Lie ideals in associative algebras, Can.
Math. Bull., \textbf{27}(1984), 10-15.

\bibitem{Sh} V.S.Shulman, Review on the book of K.Davidson "Nest
algebras", Algebra and Analysis, \textbf{2} (1990), 236-255.

\bibitem{Sh90} V.S.Shulman,  Spectral
synthesis and the Fuglede-Putnam-Rosenblum theorem. (Russian) Teor. Funktsij Funktsional. Anal. i Prilozhen. No.
54 (1990), 25--36; translation in J. Soviet Math. 58 (1992), no. 4, 312--318.

\bibitem{ST} V. Shulman, L. Turowska, Operator synthesis. I.
Synthetic sets, bilattices and tensor algebras. J. Funct. Anal.
\textbf{209} (2004), no. 2, 293--331.

\bibitem{Var} N. Varopoulos, Tensor algebras and harmonic analysis,
Acta Math. \textbf{119} (1967), 51-112.


\end{thebibliography}
\end{document}